\documentclass[reqno,a4paper,11pt]{amsproc}
\usepackage{amsmath,amscd,amsfonts,amssymb}
\usepackage{mathrsfs,dsfont}

\def\R{\mathbb{R}}

\def\Z{\mathbb{Z}}

\def\T{\mathbb{T}}

\def\1{\mathds{1}}
\def\eps{\varepsilon}

\def\zftc{Z_{\ft{\c}}}

\def\zf{Z_f}
\def\zg{Z_g}

\def\c{\mathbf{c}}
\def\k{\mathbf{k}}

\renewcommand\leq{\leqslant}
\renewcommand\geq{\geqslant}

\renewcommand\Re{\operatorname{Re}}

\newcommand{\ft}[1]{\widehat #1}
\newcommand\dotprod[2]{\langle #1 , #2 \rangle}
\newcommand\supp{\operatorname{supp}}
\newcommand\clos{\operatorname{Clos}}
\newcommand\sign{\operatorname{sign}}
\newcommand\lip{\operatorname{Lip}}
\newcommand\E{\operatorname{\mathbb{E}}}

\theoremstyle{plain}
\newtheorem*{theorem*}{Theorem}
\newtheorem{theorem}{Theorem}
\newtheorem{lemma}{Lemma}
\newtheorem{corollary}{Corollary}
\newtheorem*{theorem-a}{Theorem A}

\theoremstyle{definition}
\newtheorem*{definition*}{Definition}
\newtheorem*{remark*}{Remark}

\newenvironment{enumerate-math}
{\begin{enumerate}
\addtolength{\itemsep}{5pt}
\renewcommand\theenumi{(\roman{enumi})}}
{\end{enumerate}}

\begin{document}

\title[Wiener's `closure of translates' problem]
{Wiener's `closure of translates' problem\\
and Piatetski-Shapiro's\\
uniqueness phenomenon}

\author{Nir Lev}
\address{Department of Mathematics, Weizmann Institute of Science, Rehovot 76100, Israel.}
\email{nir.lev@weizmann.ac.il}

\author{Alexander Olevskii}
\address{School of Mathematical sciences, Tel-Aviv University, Tel-Aviv 69978, Israel.}
\email{olevskii@post.tau.ac.il}

\begin{abstract}
N.\ Wiener characterized the cyclic vectors (with respect to trans\-lations)
in $\ell^p(\Z)$ and $L^p(\R)$, $p=1,2$, in terms of the zero set of the
Fourier transform. He conjectured that a similar characterization should
be true for $1 < p < 2$. Our main result contradicts this conjecture.
\end{abstract}

\thanks{Research supported in part by the Israel Science Foundation.}

\maketitle


\section{Introduction}

\subsection{}
Let $G$ be a locally-compact abelian group, and $1 \leq p < \infty$.
A function $f \in L^p(G)$ is called a \emph{cyclic vector} (with respect
to translations) if the linear span of its translates is dense in the
space. It is well known that $f \in L^p(\T)$ (where $\T$ is the circle group)
is a cyclic vector if and only if all the Fourier coefficients of $f$ are
non-zero. The same is true for general compact groups (see \cite{segal:group}).

In the non-compact case the situation is more complicated.
N.\ Wiener \cite{wiener:tauberian} characterized the cyclic vectors in
$L^p(\R)$ (or $\ell^p(\Z)$) only for $p=1$ and $2$. We formulate
the result for $\ell^p(\Z)$, the $L^p(\R)$ case is similar.

\begin{theorem-a}[Wiener]
Let $\c = \{c_n\}$, $n \in \Z$.
\begin{enumerate-math}
\item
$\c$ is a cyclic vector in $\ell^2(\Z)$ if and only if the Fourier transform
\begin{equation}
\label{eq:ft-c}
\ft{\c}(t) := \sum_{n \in \Z} c_n e^{int}
\end{equation}
is non-zero almost everywhere.
\item
$\c$ is cyclic in $\ell^1(\Z)$ if and only if $\ft{\c}(t)$ has no zeros.
\end{enumerate-math}
\end{theorem-a}

Part (i) is a consequence of the unitarity of the Fourier transform.
Part (ii) is more delicate, the proof is based on the fact that the
space $\ell^1(\Z)$ is a convolution algebra.

In both cases the result can be stated as follows: $\c$ is a cyclic
vector if and only if the set
\[
\zftc := \{t \in \T \, : \, \ft{\c}(t) = 0\}
\]
of the zeros of the Fourier transform \eqref{eq:ft-c} is ``small'' in a
certain sense. Wiener conjectured (see \cite{wiener:tauberian}, p. 93)
that a similar result should be true for $\ell^p$ spaces, at least for
$1 < p < 2$. This problem has been studied by Segal \cite{segal:span},
Beurling \cite{beurling:closure}, Pollard \cite{pollard}, Herz \cite{herz},
Newman \cite{newman} and other authors (see \cite{edwards, kinukawa,
rosenblatt-shuman}).

First of all, one should define precisely how to understand the zero set.
The answer is obvious if the vector $\c$ is assumed to be in $\ell^1(\Z)$
(or $L^1(\R)$). The above mentioned authors have studied the problem
under this assumption. We shall keep it as well.

A more serious question is -- what kind of ``smallness'' one should
consider? One can show (see \cite{segal:span}) that if $1<p<2$ then the
condition in part (i) of Theorem A is not sufficient, and the condition
in part (ii) is not necessary for cyclicity in $\ell^p$. So one should
look for an ``intermediate measurement'' of smallness.

A.\ Beurling proved in \cite{beurling:closure} that if the Hausdorff
dimension of $\zftc$ is less than $2(p-1)/p$ then $\c$ is a cyclic vector
in $\ell^p$ $(1<p<2)$. This condition is sharp, but it is not necessary for the
cyclicity (see \cite{newman}).

On the other hand, it is well known that not only metrical but also arithmetical
``thinness'' properties may play an important role in problems of harmonic
analysis. In the above cited papers, various metrical and non-metrical properties
of the zero set of cyclic vectors in $L^p(\R)$ ($\ell^p(\Z)$) have been studied,
and a number of interesting results were obtained. In particular, for $p>2$ the
cyclic vectors were indeed characterized by a condition in terms of the zero set.
This condition (see Section \ref{subsection:background} below) is not easy to
check, but anyway it supports Wiener's conjecture that cyclicity depends on the
set $\zftc$ only.

Our main result shows that this is not the case for $1<p<2$.

\begin{theorem}
\label{thm:main-cyclic}
Let $1 < p < 2$. Then there is a compact set $K$ on the circle $\T$
with the following properties:
\begin{enumerate-math}
\renewcommand\theenumi{(\alph{enumi})}
\item
\label{cond:a}
If a vector $\c$ has fast decreasing coordinates, say
$\sum_{n \in \Z} |c_n| \, |n|^\eps < \infty$ for some
$\eps > 0$, and $\ft{\c}$ vanishes on $K$, then $\c$
is not cyclic in $\ell^p(\Z)$.
\item
\label{cond:b}
There exists $\c \in \ell^1(\Z)$, such that $\ft{\c}$
vanishes on $K$, and $\c$ is a cyclic vector in $\ell^p(\Z)$.
\end{enumerate-math}
\end{theorem}

It follows that no characterization of the cyclic vectors exists
in terms of the zeros of the Fourier transform:

\begin{corollary}
\label{cor:impossible-integers}
Given any $p$, $1 < p < 2$, one can find two vectors in $\ell^1(\Z)$
such that one is cyclic in $\ell^p(\Z)$ and the other is not,
but their Fourier transforms have an identical set of zeros.
\end{corollary}

A similar result is true for $L^p(\R)$ (see Section \ref{section:non-periodic} below).

\subsection{}
Our approach to the problem is based on its relation to the uniqueness
problem in Fourier analysis, or more specifically, to an aspect of it
which we call the ``Piatetski-Shapiro phenomenon''.

Recall that a set $K \subset \T$ is called a
\emph{set of uniqueness} (U-set) if whenever a trigonometric series
\[
\sum_{n \in \Z} c_n e^{int}
\]
converges to zero at every point $t \notin K$, then all the coefficients
$c_n$ must be zero. Otherwise, $K$ is called a \emph{set of multiplicity} (M-set).
Classical Riemannian theory allows one to characterize the compact M-sets
as the compacts which support a non-zero distribution $S$ with Fourier
transform $\ft{S}(n)$ tending to zero as $|n| \to \infty$ (see \cite{kahane-salem}).

It was D.\ E.\ Menshov who discovered (1916) that a set of Lebesgue
measure zero can be an M-set. In fact, Menshov constructed a compact set
$K$ of Lebesgue measure zero, which supports a \emph{measure} whose Fourier
transform vanishes at infinity (see \cite{bary}). It was believed within
a long time that every compact M-set must support such a measure. This was
disproved in 1954 by I.\ Piatetski-Shapiro \cite{piatetski}, who constructed
a compact M-set which does not support such a measure.

This striking result was further developed by T.\ K\"{o}rner
\cite{korner} and R.\ Kaufman \cite{kaufman}, who presented different
examples of compact M-sets $K$ which are Helson sets. The latter
means that every measure $\mu$ supported by $K$ satisfies the condition
\[
\limsup_{|n| \to \infty} |\ft{\mu}(n)| \geq \delta(K) \int |d\mu|,
\quad \delta(K) > 0.
\]

As Kaufman mentioned, his construction was inspired by
Piatetski-Shapiro's original ideas. An additional improvement of this
technique was done by J.-P.\ Kahane \cite[pp. 213--216]{kahane-salem}
in his presentation of Kaufman's paper. This presen\-tation
was the starting point for our approach.

\subsection{}
Let $X$ be a Banach space of sequences (with a norm weaker than $\ell^2$).
We say that \emph{Piatetski-Shapiro's phenomenon exists for the space $X$}
if there is a compact set $K \subset \T$, which supports a (non-zero)
distribution $S$ with Fourier transform $\ft{S} \in X$, but which does
not support such a measure. The result from \cite{piatetski} means that
the phenomenon exists for the space $c_o$. On the other hand, potential
theory (see \cite{beurling:spectres, kahane-salem}) provides an important
example of spaces for which the phenomenon does not exist: the weighted
spaces $\ell^2(\Z, w)$ with e.g.\ the power weight
$w(n) = (1+|n|)^{-\alpha}$, $0 < \alpha \leq 1$.

Our concern, inspired by Wiener's problem, was:
\begin{quote}
\emph{Does Piatetski-Shapiro's phenomenon exist for $\ell^q$ spaces, $q > 2$}\,?
\end{quote}
The answer is \emph{yes}:

\begin{theorem}[\cite{lev-olevskii:piatetski}]
\label{thm:piatetski-lq}
Given any $q > 2$ there is a compact set $K \subset \T$ such that
\begin{enumerate-math}
\renewcommand\theenumi{(\alph{enumi}\:\!$'$)}
\item
\label{cond:atag}
$K$ supports a non-zero distribution $S$ such that
$\ft{S} \in \ell^q$;
\item
\label{cond:btag}
$K$ does not support any non-zero measure $\mu$ such
that $\ft{\mu} \in \ell^q$.
\end{enumerate-math}
\end{theorem}

The role of Theorem \ref{thm:piatetski-lq} in the cyclicity
problem is clarified by the following observation (see Section
\ref{subsection:background} below):

\emph{The condition \ref{cond:atag} is equivalent to part
\ref{cond:a} of Theorem \ref{thm:main-cyclic}, while the
condition \ref{cond:btag} is necessary for part \ref{cond:b},
with $q = p/(p-1)$.}

So Theorem \ref{thm:piatetski-lq} provides a chance for
(although it does not imply) the existence of a counter-example
to Wiener's conjecture.
Such a counter-example -- in a weaker form than Corollary
\ref{cor:impossible-integers} above -- was sketched in
\cite{lev-olevskii:generators}. The present paper contains
full proofs and extensions of the results obtained in
\cite{lev-olevskii:piatetski,lev-olevskii:generators}.
It is organized as follows. 

In Section \ref{section:auxiliary} we give
some preliminary background and auxiliary lemmas.

Section \ref{section:riesz-products} is the key one in the paper.
Our main tools are special measures on the circle, defined by
Riesz-type products, and a version of Bernstein stochastic
exponential estimate.

In Section \ref{section:helson} we construct a Helson set with the
property \ref{cond:atag} above. In Section \ref{section:proof-main}
we prove that every Helson set admits a vector with the property
\ref{cond:b}. So Theorem \ref{thm:main-cyclic} follows.

The non-periodic version is considered in Section \ref{section:non-periodic}.
The last Section \ref{section:remarks} contains some additional remarks.
In particular we discuss there the relation of Theorem \ref{thm:main-cyclic}
to P.\ Malliavin's celebrated ``non-synthesis'' phenomenon, and mention some
open problems.


\section{Preliminaries. Lemmas.}
\label{section:auxiliary}

\subsection{Notation}
In what follows $\T$ is the circle group $\R / 2 \pi \Z$.

As usual, $C(\T)$ is the space of continuous complex functions on $\T$,
with the norm $\|f\|_\infty := \sup |f(t)|$, $t \in \T$.
By a ``measure'' on $\T$ we always mean an element of the dual space
$M(\T)$, that is, a finite complex Borel measure.

We denote by $\{\ft{S}(n)\}$, $n \in \Z$, the Fourier coefficients
of a Schwartz distribution $S$ on $\T$. It will also be convenient
to keep the notation $\ft{\c}$ for the Fourier transform of a vector
$\c \in \ell^1(\Z)$ as defined in \eqref{eq:ft-c}.

Let $A_p(\T)$, $1 \leq p \leq \infty$, denote the Banach space of
distributions $S$ on $\T$ with Fourier coefficients belonging to
$\ell^p(\Z)$, endowed with the norm
\[
\|S\|_{A_p} := \|\ft{S}\|_{\ell^p}.
\]
For $p = 1$ this is the Wiener algebra $A(\T)$ of absolutely convergent
Fourier series (see \cite{kahane:absolument}).
Throughout we will use the following standard properties,
\[
\|f\|_\infty \leq \|f\|_A, \quad f \in A(\T),\\[2pt]
\]
\[
\|f \cdot g\|_{A_p} \leq \|f\|_{A} \, \|g\|_{A_p}, 
\quad f \in A(\T), \; g \in A_p(\T).
\]

\subsection{Cyclic vectors}
\label{subsection:background}

In this section we refer to some basic results about cyclicity.
The results go back to Segal \cite{segal:group}, Beurling
\cite{beurling:closure}, Pollard \cite{pollard}, Herz \cite{herz}
and Newman \cite{newman}. Actually the first four authors considered
Wiener's problem in $L^p(\R)$ rather than $\ell^p(\Z)$, but,
as the last author mentioned, ``the distinction is not vital''. 
See also Kahane-Salem \cite{kahane-salem}, pp. 111--112 and 122--123.

First it would be convenient to reformulate the concept of cyclicity
in an equivalent way, using the following

\begin{definition*}
An element $f \in A_p(\T)$, $1 \leq p < \infty$, is called a cyclic vector
(with respect to multiplication by exponentials) if the set $\{P(t) f(t)\}$,
where $P$ goes through all trigonometric polynomials, is dense in $A_p(\T)$.
\end{definition*}

Clearly, a vector $\c$ is cyclic (with respect to translations) in $\ell^p(\Z)$
if and only if its Fourier transform $f := \ft{\c}$ is cyclic in $A_p(\T)$ in
the sense just defined.

In what follows $f$ is assumed to belong to the Wiener algebra $A(\T)$, $\zf$ denotes
the set of the zeros of $f$, and $q = p/(p-1)$.
\begin{enumerate-math}
\item
\label{remark:polynom-cyclic}
\emph{$f$ is cyclic in $A_p(\T)$ if and only if there is a
sequence of trigonometric polynomials $P_n$ such that}
\[
\lim_{n \to \infty}
\|1 - P_n \cdot f\|_{A_p} = 0.
\]
\item
\label{remark:finite-cyclic}
\emph{If $\zf$ is finite then $f$ is cyclic in $A_p(\T)$ for every $p > 1$.}
\item
\label{remark:noncyclic-distribution}
\emph{If $f$ is a non-cyclic vector in $A_p(\T)$ then there is a
non-zero distribution $S \in A_q(\T)$, supported by $\zf$.}
\item
\label{remark:measure-noncyclic}
\emph{If there is a non-zero measure $\mu \in A_q(\T)$,
supported by $\zf$, then $f$ is a non-cyclic vector in $A_p(\T)$.}
\item
\label{remark:smooth-distribution-noncyclic}
\emph{If $f$ is continuously differentiable, and there is a non-zero
distribution $S \in A_q(\T)$ supported by $\zf$, then f is a non-cyclic
vector in $A_p(\T)$.}
\end{enumerate-math}

Actually (see \cite{herz}) the smoothness condition in
\ref{remark:smooth-distribution-noncyclic} can be reduced up to
$f \in \lip \eps$ for some $\eps > 0$, or, in terms of the Fourier
coefficients of $f$, up to
\[
\sum_{n \in \Z} |\ft{f}(n)| \, |n|^\eps < \infty
\quad \text{for some $\eps > 0$.}
\]

Observe that if $p > 2$ then $A_q(\T)$ is a functional space
(embedded in $L^2(\T)$). So the conditions
\ref{remark:noncyclic-distribution}--\ref{remark:measure-noncyclic}
imply:

\emph{A function $f \in A(\T)$ is a cyclic vector in $A_p(\T)$,
$p > 2$, if and only if its zero set $\zf$ does not support any
non-zero function $g \in A_q(\T)$.}

This condition is not very effective, but it shows that the cyclicity
of a vector $c \in \ell^1(\Z)$ in the space $\ell^p(\Z)$, $p>2$, admits 
characterization in terms of the zero set of the Fourier transform
(as Wiener thought). Here, by the way, the condition that $\zf$ has
Lebesgue measure zero obviously implies the cyclicity, but not vice
versa \cite{newman} (see also \cite{katznelson}, pp. 101--102).

The case $p = \infty$ was also considered. D.\ Newman \cite{newman}
proved that $\c \in \ell^1(\Z)$ is a cyclic vector in the space $c_0(\Z)$
if and only if $\zftc$ is a nowhere dense set in $\T$.

Now let $1 < p < 2$ (the case where Wiener's conjecture was most
certain). Then $A_q(\T)$ is not a functional space, and
\ref{remark:noncyclic-distribution}--\ref{remark:smooth-distribution-noncyclic}
only imply the following:

\emph{Let $f \in A(\T)$ be smooth $(f \in \lip \eps, \, \eps > 0)$.
Then it is cyclic in $A_p(\T)$, $1<p<2$, if and only if $\zf$ does
not support any non-zero distribution $S \in A_q(\T)$.}

However we will see that without the smoothness condition, the zero set
$\zf$ does not provide a characterization of the cyclic vectors.


\subsection{Auxiliary polynomials}
\label{subsection:auxiliary}
We shall use trigonometric polynomials with the following properties.

\begin{lemma}
\label{lemma:auxiliary}
Given any $q > 2$ and $\gamma > 0$ there is a real trigonometric
poly\-nomial $\varphi = \varphi_{q, \gamma}$ such that
\begin{equation}
\label{eq:auxiliary}
\ft{\varphi}(0) = 0, \quad
\|\varphi\|_{\infty} \leq 1, \quad
\|\varphi\|_{L^2} = \tfrac1{2}, \quad
\|\varphi\|_{A_q} < \gamma.
\end{equation}
\end{lemma}

Here and below $\|\cdot\|_{L^2}$ denotes the $L^2$ norm on $\T$ with
respect to the normalized Lebesgue measure.

There are several ways to get Lemma \ref{lemma:auxiliary}.
In particular one may use the Shapiro-Rudin polynomials
(see \cite{kahane:absolument}, p. 52). 
Namely, for an appropriate choice of signs $\eps_n = \pm 1$
$(n=1,2,\dots)$ the trigonometric polynomial
\[
Q_k(t) = \sum_{n=1}^{2^k} \eps_n \cos nt
\]
satisfies $\|Q_k\|_\infty \leq 2^{(k+1)/2}$.
It follows that if $k = k(q, \gamma)$ is sufficiently large, then
\[
\varphi(t) := 2^{-(k+1)/2} \, Q_k(t)
\]
is a real trigonometric polynomial with the properties \eqref{eq:auxiliary}.


\subsection{Kahane's lemma}
\label{subsection:kahane}
One of the key arguments in \cite{piatetski} and \cite{kaufman}
is based on the uniqueness theorem for power series. In Kahane's
presentation of Kaufman's paper (see \cite{kahane-salem},
pp. 213--216) this point was performed as follows:

\emph{Given any $\delta > 0$ there is a real, signed measure $\rho$,
supported by a finite subset of the interval $(1 - \delta, \, 1)$, such that
\begin{equation}
\label{eq:kahane-constraints}
\int d\rho = 1 \qquad \text{and} \qquad
\Big| \int s^k \, d\rho(s) \Big| < \delta
\quad (k=1,2,\dots).
\end{equation}}

This lemma was proved in \cite[p. 214]{kahane-salem} based on the
Hahn-Banach theorem. Here we shall need a quantitative version, with
an estimate on the total variation of the measure.

\begin{lemma}
\label{lemma:kahane-lemma}
Let an interval $I = (a,b)$, $0 < a < b < \frac1{2}$,
and $0 < \delta < 1$ be given. Then there is a real, signed
measure $\rho$, supported by a finite subset of $I$,
such that \eqref{eq:kahane-constraints} holds, and such
that
\begin{equation}
\label{eq:kahane-estimate}
\int |d\rho| < \delta^{\, - c(I)},
\end{equation}
where $c(I) > 0$ is a constant which depends only on $I$.
\end{lemma}

We proceed to the proof of Lemma \ref{lemma:kahane-lemma}.

\subsubsection{The measure}
Given $n$ distinct points $s_1, \dots, s_n \in I$, consider a measure
$\rho$ supported by these points and defined uniquely by the condition
\[
\int p(s) \, d\rho(s) = p(0), \quad \,
\text{for every algebraic polynomial $p$ of degree $\leq n-1$.}
\]
In particular,
\begin{equation}
\label{eq:kahane-moments-zero}
\int d\rho = 1, \qquad
\int s^k \, d\rho(s) = 0 \quad (k=1,2,\dots,n-1).
\end{equation}

Given any function $f(s)$ one has $\int f(s) \, d\rho(s) = p(0)$,
where $p$ is the unique polynomial of degree $\leq n-1$ which
interpolates $f$ at the nodes $s_1, \dots, s_n$. It is well known
(see for example \cite{atkinson}, pp. 134--135) that if $f(s)$ is
real-valued and sufficiently smooth, then there is $0 \leq \xi < b$
such that
\[
f(0) = p(0) + \frac{f^{(n)}(\xi)}{n!} \, \prod_{j=1}^{n} (0 - s_j).
\]
Applying this with $f(s) = s^k$, $k \geq n$, gives
\[
\int s^k \, d\rho(s) = (-1)^{n-1} \, \binom{k}{n} \, \xi^{k-n}
\, \prod_{j=1}^{n} s_j \, ,
\]
and consequently the moments of $\rho$ satisfy the estimate
\begin{equation}
\label{eq:kahane-moments-bound}
\Big| \int s^k \, d\rho(s) \Big| < 2^k \cdot b^{k - n}
\cdot b^n = (2b)^k \leq (2b)^n \quad (k \geq n).
\end{equation}

\subsubsection{The total variation}
Using the Lagrange polynomials
\[
l_j(s) = \prod_{i \neq j} \frac{s - s_i}{s_j - s_i}
\quad (1 \leq j \leq n)
\]
one can calculate the masses
\[
\rho(\{s_j\}) = \int l_j(s) \, d\rho(s) = l_j(0)
= \prod_{i \neq j} \frac{s_i}{s_i - s_j} \, .
\]
We choose the points $s_1, \dots, s_n$ as equally spaced nodes,
$s_j = a + (j - \tfrac1{2})h$ where $h = (b-a)/n$.
Then
\[
\big| \, \rho(\{s_j\}) \, \big| = 
\frac1{h^{n-1} \, (j-1)! \, (n-j)!}
\prod_{i \neq j} s_i \, ,
\]
and so we have
\begin{align}
\int |d \rho|
&\leq \Big(\frac{b}{h} \Big)^{n-1} \sum_{j=1}^{n} \frac1{(j-1)! \, (n-j)!}
= \frac1{n!} \Big( \frac{nb}{b-a} \Big)^{n-1} \sum_{j=1}^{n} j \binom{n}{j}
\nonumber \\
\label{eq:kahane-variation-estimate}
&= \frac{n^n}{n!} \Big( \frac{2b}{b-a} \Big)^{n-1}
\leq \Big( \frac{2eb}{b-a} \Big)^{n-1}.
\end{align}

Finally, choose $n$ to be the least integer 
$\geq \frac{\log(1/\delta)}{\log(1/2b)}$. It follows from
\eqref{eq:kahane-moments-zero}, \eqref{eq:kahane-moments-bound}
and \eqref{eq:kahane-variation-estimate} that $\rho$ satisfies
both \eqref{eq:kahane-constraints} and \eqref{eq:kahane-estimate}.
This proves the lemma. \qed

\begin{remark*}
One can show that a power estimate \eqref{eq:kahane-estimate}
in Lemma \ref{lemma:kahane-lemma} is sharp.
\end{remark*}


\subsection{Bernstein inequality}
\label{subsection:bernstein}
Bernstein exponential estimates for sums of independent random
variables are classical. Different versions, adopted for sums
of ``almost'' independent variables, in various senses, are also
well known.

In particular, Azuma \cite{azuma} considered the so-called 
multiplicatively orthogonal systems and obtained Bernstein-type
exponential estimates for them.

It will be convenient for us to consider a similar version, suitable
for an ``almost multiplicative'' system of random variables, in the
following sense:

\begin{lemma}
\label{lemma:bernstein-epsilon}
Let $X_1, \dots, X_N$ be random variables on a probability
space $(\Omega, P)$ such that $-1 \leq X_j \leq 1$ $(j=1,2,\dots,N)$.
Suppose that
\begin{equation}
\label{eq:equal-expectations}
\E(X_1) = \cdots = \E(X_N) = \mu > 0,
\end{equation}
and that there is $0 < \eps < 1$ such that
\begin{equation}
\label{eq:almost-independent}
(1 - \eps) \, \mu^{|A|} \leq \E \Big\{ \prod_{j \in A}
X_j \Big\} \leq (1 + \eps) \, \mu^{|A|}
\end{equation}
for every non-empty subset $A \subset \{1,2,\dots,N\}$,
where $|A|$ denotes the number of elements in $A$. Define
\[
X = \frac1{N} \sum_{j=1}^{N} X_j \, .
\]
Then for any $\alpha > 0$,
\begin{equation}
\label{eq:bernstein}
P \big\{ X < \E(X) - \alpha \big\} \leq
\exp \big(- \tfrac1{8} \alpha^2 N \big)
+ \eps \exp \big( \tfrac1{4} N \big).
\end{equation}
\end{lemma}

\begin{proof}
Fix $\lambda > 0$. By the classical Bernstein method
we can estimate the probability on the left hand side
of \eqref{eq:bernstein} as follows,
\begin{equation}
\label{eq:bernstein-estimate-1}
P \big\{ X < \mu - \alpha \big\} = P \Big\{ \prod_{j=1}^{N}
e^{\lambda (\mu - X_j)} > e^{\alpha \lambda N} \Big\}
\leq e^{-\alpha \lambda N} \, \E \prod_{j=1}^{N}
e^{\lambda (\mu - X_j)}.
\end{equation}
To estimate the expectation on the right hand side we adopt
the approach of \cite{azuma}. Since $|\mu - X_j| \leq 2$
and by the convexity of the exponential function, we have
\[
e^{\lambda (\mu - X_j)} 
\leq \cosh(2 \lambda) + ((\mu - X_j)/2) \, \sinh(2 \lambda)
= b - a X_j,
\]
where
\[
a := (1/2) \, \sinh(2 \lambda) \quad \text{and} \quad
b := \cosh(2\lambda) + (\mu/2) \, \sinh(2 \lambda).
\]
It follows that
\begin{equation}
\label{eq:bernstein-open-bracket}
\E \prod_{j=1}^{N} e^{\lambda (\mu - X_j)}
\leq \E \prod_{j=1}^{N} (b - a X_j) 
= \sum_{A \subset \{1,\dots,N\}}
(-a)^{|A|} \, b^{N-|A|} \, \E \prod_{j \in A} X_j.
\end{equation}
Now we invoke the assumption \eqref{eq:almost-independent}, which 
(together with the fact that $a, b$ are positive numbers) implies
that the right hand side of \eqref{eq:bernstein-open-bracket} is
not larger than
\[
\sum (-a)^{|A|} \, b^{N-|A|} \, 
\Big\{ 1 + (-1)^{|A|} \, \eps \Big\} \, \mu^{|A|}\\
= (b - a \mu)^N + \eps (b + a \mu)^N.
\]
It is easy to see that $b - a \mu \leq \exp (2 \lambda^2)$ and
$b + a \mu \leq \exp(2 \lambda)$, so it follows that
\begin{equation}
\label{eq:bernstein-estimate-2}
\E \prod_{j=1}^{N} e^{\lambda (\mu - X_j)} \leq 
\exp (2 \lambda^2 N) + \eps \, \exp(2 \lambda N).
\end{equation}
Finally, a combination of \eqref{eq:bernstein-estimate-1} and
\eqref{eq:bernstein-estimate-2}, with $\lambda = \alpha / 4$,
gives
\[
P \big\{ X < \mu - \alpha \big\} 
\leq \exp \big(- \tfrac1{8} \alpha^2 N \big) + 
\eps \exp \big( \tfrac1{4} \alpha (2 - \alpha) N \big).
\]
However $\alpha (2 - \alpha) \leq 1$, so the estimate
\eqref{eq:bernstein} follows.
\end{proof}


\section{Riesz-type products and exponential estimates}
\label{section:riesz-products}

This section contains the central part of our approach.
We prove here the following main lemma:

\begin{lemma}
\label{lemma:principal}
Suppose that we are given numbers $q > 2$ and $\eps > 0$, and a real
trigonometric polynomial $u$, not identically zero. Then we can
find a compact set $K$ (a finite union of segments), an infinitely
smooth function $f$ and a real trigonometric polynomial $P$ such that
\begin{enumerate-math}
\item
\label{cond:principal-1}
$f$ is supported by $K$, \, $\|1 - f\|_{A_q} < \eps$,
\item
\label{cond:principal-2}
$\displaystyle \inf_{t \in K} |P(t)| > 1$, \,
$P(t) u(t) > 0$ on $K$, \, $\|P\|_A \leq C(q)$,
\end{enumerate-math}
where $C(q)$ is a constant which depends only on $q$.
\end{lemma}

The proof involves several steps.


\subsection{Multiplicativity}
We start with the following simple property.

\begin{lemma}
\label{lemma:multiplicativity}
Let $\nu$ be a positive integer, and suppose that $P_j$ are
trigonometric polynomials, $\deg P_j < \nu$ $(j=0,1,\dots,N)$.
Then
\[
\int_{\T} \Big\{ \prod_{j=0}^{N} P_j(\nu^j t) \Big\} \, \frac{dt}{2\pi}
= \prod_{j=0}^{N} \Big\{ \int_{\T} P_j(t) \, \frac{dt}{2\pi} \Big\}.
\]
\end{lemma}

\begin{proof}
By Fourier expansion, the left hand side is equal to
\[
\sum_{\k} \Big\{ \prod_{j=0}^{N} \ft{P}_j(k_j) \Big\} \int_{\T}
e^{i (k_0 + k_1 \nu + k_2 \nu^2 + \cdots + k_N \nu^N)t}
\, \frac{dt}{2\pi} \, ,
\]
where the sum goes through all integer vectors
$\k = (k_0, k_1, \dots, k_N)$ such that $|k_j| \leq \deg P_j$.
However it is easy to check that the only solution of the equation
\[
k_0 + k_1 \nu + k_2 \nu^2 + \cdots + k_N \nu^N = 0
\]
with $\k$ as above, is $\k=(0,0,\dots,0)$. This implies the result.
\end{proof}


\subsection{Riesz-type measures}
Suppose that we are given a positive integer $N$, a real
trigonometric polynomial $\varphi$ with the properties
\begin{equation}
\label{eq:phi-properties}
\ft{\varphi}(0) = 0, \quad \|\varphi\|_\infty \leq 1,
\end{equation}
and also a real trigonometric polynomial $w$ such that
\begin{equation}
\label{eq:w-property}
\|w\|_\infty \leq 1.
\end{equation}
Choose a large integer $\nu$, satisfying the condition
\begin{equation}
\label{eq:nu-properties}
\nu > 2 \max \{ \deg \varphi , \, N \deg w \},
\end{equation}
and define a ``Riesz-type product''
\begin{equation}
\label{eq:riesz-prod}
\lambda_s(t) = \prod_{j=1}^{N} \Big( 1 + s \, w(t) \, \varphi(\nu^j t) \Big),
\quad 0 < s < 1.
\end{equation}

Introduce a measure $\mu_s$ on the circle $\T$,
\begin{equation}
\label{eq:def-mu-s}
d\mu_s(t) = \lambda_s(t) \, \frac{dt}{2\pi} \, .
\end{equation}
Observe first that it is a probability measure on $\T$.
Indeed, it is clear from the properties above that
$\lambda_s$ is everywhere positive. Now expand the product
\eqref{eq:riesz-prod} into the form
\begin{equation}
\label{eq:riesz-prod-open}
\lambda_s(t) = 1 + \sum_{B}
\big( s \, w(t) \big)^{|B|} \prod_{j \in B} \varphi(\nu^j t),
\end{equation}
where the sum goes through all non-empty subsets $B \subset
\{1,\dots,N\}$. The condition \eqref{eq:nu-properties} allows
one to use Lemma \ref{lemma:multiplicativity}, which implies that
\[
\int_{\T} \lambda_s(t) \, \frac{dt}{2\pi} = 1 +
\sum_{B} \Big\{ \int_{\T} w(t)^{|B|} \, \frac{dt}{2\pi} \Big\}
\Big\{ s \int_{\T} \varphi(t) \, \frac{dt}{2\pi} \Big\}^{|B|}.
\]
However all terms in the above sum are zero, since $\ft{\varphi}(0) = 0$.
So it follows that
\[
\int_{\T} \lambda_s(t) \, \frac{dt}{2\pi} = 1,
\]
and this proves the claim.

\subsection{Random variables}
Consider random variables defined by
\begin{equation}
\label{eq:random-variables}
X_j(t) = w(t) \, \varphi(\nu^j t), \quad 1 \leq j \leq N,
\end{equation}
on the probability space $(\T, \mu_s)$.

It is well known that these variables are ``almost independent'' with
respect to the Lebesgue measure on $\T$. However, we will see that
(under some additional condition) they are ``almost independent'' also 
with respect to $\mu_s$, which is going to be essentially ``singular''
with respect to the Lebesgue measure.

To establish such a property we first compute the ``multiplicative moments''.

\begin{lemma}
\label{lemma:product-rule}
Let $A$ be a non-empty subset of $\{1,2,\dots,N\}$. Then
\begin{equation}
\label{eq:product-rule}
\E \Big\{ \prod_{j \in A} X_j \Big\} =
\Big\{ s \int_{\T} \varphi(t)^2 \, \frac{dt}{2\pi} \Big\}^{|A|}
\Big\{ \int_{\T} w(t)^{2|A|} \, \frac{dt}{2\pi} \Big\}.
\end{equation}
\end{lemma}

\begin{proof}
By \eqref{eq:def-mu-s} and \eqref{eq:random-variables},
the left hand side of \eqref{eq:product-rule} is equal to
\begin{equation}
\label{eq:product-rule-before}
\int_{\T} \Big\{ w(t)^{|A|} \prod_{j \in A} \varphi(\nu^j t)
\Big\} \, \lambda_s(t) \, \frac{dt}{2\pi} \, .
\end{equation}
Let us again consider the expansion \eqref{eq:riesz-prod-open} for
$\lambda_s$, however this time we do not distinguish the constant
term as before, but rather write
\[
\lambda_s(t) = \big( s \, w(t) \big)^{|A|} \prod_{j \in A} \varphi(\nu^j t)
+ \cdots
\]
where the implicit terms correspond to all subsets $B \subset
\{1,\dots,N\}$ which are different from $A$. Inserting this
expression into \eqref{eq:product-rule-before} one can see
that the integration of the explicit term gives
\[
s^{|A|} \int_{\T} \Big\{ w(t)^{2|A|} 
\prod_{j \in A} \varphi^2(\nu^j t) \Big\} \frac{dt}{2\pi}
\]
which, by the condition \eqref{eq:nu-properties} and Lemma
\ref{lemma:multiplicativity}, provides the right hand side of
\eqref{eq:product-rule}. So to conclude the proof it is enough
to show that the integrals of the other terms in the sum are all zero.

Indeed, if $B$ is any subset $\neq A$ then the corresponding term is
\[
s^{|B|} \int_{\T} \Big\{
w(t)^{|A|+|B|} \prod_{j \in A \triangle B} \varphi(\nu^j t)
\prod_{j \in A \cap B} \varphi^2(\nu^j t) \Big\} \frac{dt}{2\pi}
\]
which, again by \eqref{eq:nu-properties} and Lemma
\ref{lemma:multiplicativity}, is equal to
\[
s^{|B|}
\Big\{ \int_{\T} w(t)^{|A|+|B|} \, \frac{dt}{2\pi} \Big\}
\Big\{ \int_{\T} \varphi(t) \, \frac{dt}{2\pi} \Big\}^{|A \triangle B|}
\Big\{ \int_{\T} \varphi^2(t) \, \frac{dt}{2\pi} \Big\}^{|A \cap B|}.
\]
However this is zero, because $\ft{\varphi}(0) = 0$, so the
lemma is proved.
\end{proof}

One can see that if the trigonometric polynomial $w$ is mostly close to
$1$ in modulus, then the integrals of the even powers of $w$ which appear
in \eqref{eq:product-rule} are almost equal to $1$. We will see that in
such a case the $X_1, \dots, X_N$ form an ``almost multiplicative'' system
of random variables (in the sense of Lemma \ref{lemma:bernstein-epsilon})
with respect to the measure $\mu_s$.

Precisely, Lemma \ref{lemma:product-rule} allows one to find
the expectations
\begin{equation}
\label{eq:expectations}
\E(X_j) = s \, \|\varphi\|_{L^2}^2 \, \|w\|_{L^2}^2
\quad (1 \leq j \leq N).
\end{equation}
In particular all the $X_j$ have the same expectation, as in
\eqref{eq:equal-expectations}. Now suppose that the trigonometric
polynomial $w$ satisfies, in addition to property \eqref{eq:w-property},
also the condition
\begin{equation}
\label{cond:w-eps}
\Big\{ \int_{\T} w(t)^2 \, \frac{dt}{2\pi} \Big\}^{N} > \frac1{1 + \eps}
\qquad \text{for some $0 < \eps < 1$.}
\end{equation}
Then, given any non-empty $A \subset \{1,2,\dots,N\}$, by
\eqref{eq:product-rule} and Jensen's inequality
\[
\E \Big\{ \prod_{j \in A} X_j \Big\} \geq
\Big\{ s \int_{\T} \varphi(t)^2 \, \frac{dt}{2\pi} \Big\}^{|A|}
\Big\{ \int_{\T} w(t)^{2} \, \frac{dt}{2\pi} \Big\}^{|A|}
= \prod_{j \in A} \E (X_j).
\]
On the other hand \eqref{eq:w-property}, \eqref{eq:product-rule}
and \eqref{cond:w-eps} imply that
\[
\E \Big\{ \prod_{j \in A} X_j \Big\} \leq 
\Big\{ s \int_{\T} \varphi(t)^2 \, \frac{dt}{2\pi} \Big\}^{|A|}
\leq (1 + \eps) \prod_{j \in A} \E (X_j).
\]
This shows that the ``almost multiplicativity'' condition
\eqref{eq:almost-independent} is satisfied.

\subsection{Concentration}
Define a trigonometric polynomial
\begin{equation}
\label{eq:def-x}
X(t) = \frac1{N} \sum_{j=1}^{N} X_j(t) = 
w(t) \cdot \frac1{N} \sum_{j=1}^{N} \varphi(\nu^j t).
\end{equation}

``Almost independence'' suggests that this average is strongly concentrated
(with respect to the measure $\mu_s$) near its expectation, and the rate
of concentration is governed by the classical exponential estimates.

Indeed, assuming \eqref{eq:phi-properties},
\eqref{eq:w-property} and \eqref{cond:w-eps} one may use
Lemma \ref{lemma:bernstein-epsilon}, which implies
\begin{equation}
\label{eq:concentration-1}
\mu_s \big\{ t : X(t) < \E(X) - \alpha \big\} \leq 
\exp \big(- \tfrac1{8} \alpha^2 N \big) + \eps \exp \big( \tfrac1{4} N \big),
\quad \alpha > 0.
\end{equation}

We use this to prove the following $L^2$-concentration estimate.

\begin{lemma}
\label{lemma:concentration-l2}
Suppose that \eqref{eq:phi-properties} and \eqref{eq:w-property} hold,
and furthermore suppose that
\begin{equation}
\label{eq:phi-large-l2-norm}
\|\varphi\|_{L^2} \geq \tfrac1{2}
\end{equation}
and
\begin{equation}
\label{eq:w-large-l2-norm}
\Big\{ \int_{\T} w(t)^2 \, \frac{dt}{2\pi} \Big\}^{N} > \frac1{1 + e^{-N}} \, .
\end{equation}
Then, for every
\begin{equation}
\label{eq:s-range}
s \in I_0 := \big( \tfrac{1}{4}, \tfrac{1}{3} \big)
\end{equation}
one has
\[
\int_{\{t  : \; X(t) < c_1 \}} \lambda_s^2(t) \, \frac{dt}{2\pi}
\; < \; 2 \, e^{- c_2 N},\\[6pt]
\]
for some absolute positive constants $c_1, c_2$.
\end{lemma}

\begin{proof}
It follows from \eqref{eq:expectations}, \eqref{eq:phi-large-l2-norm},
\eqref{eq:w-large-l2-norm} and \eqref{eq:s-range} that
\[
\E(X) = s \, \|\varphi\|_{L^2}^2 \, \|w\|_{L^2}^2 > \tfrac1{100}
\quad (s \in I_0).
\]
So the estimate \eqref{eq:concentration-1} with $\eps = e^{-N}$ implies that
\begin{equation}
\label{eq:concentration-2}
\mu_s \big\{ t : X(t) < c_1 \big\} < 
2 \exp \big(- \tfrac1{8} (\tfrac1{100} - c_1)^2 N \big),
\quad 0 < c_1 < \tfrac1{100}.
\end{equation}
Using \eqref{eq:riesz-prod} we also obtain the estimate
\begin{equation}
\label{eq:estimate-lambda-x}
\lambda_s(t) \leq \exp \Big(s \, w(t) \sum_{j=1}^{N} \varphi(\nu^j t) \Big)
= \exp \big(s N X(t) \big).
\end{equation}
A combination of \eqref{eq:concentration-2} and \eqref{eq:estimate-lambda-x}
gives, for every $s \in I_0$,
\begin{align*}
\int_{\{t \; : \; X(t) < c_1\}} \lambda_s^2(t) \, \frac{dt}{2\pi}
&\leq 
\bigg( \int_{\{t \; : \; X(t) < c_1\}} \lambda_s(t) \, \frac{dt}{2\pi} \bigg)
\bigg( \sup_{\{t \; : \; X(t) < c_1\}} \lambda_s(t) \bigg)\\[4pt]
&< 2 \exp \big( - \tfrac1{8} (\tfrac1{100} - c_1)^2 N \big)
\exp \big( \tfrac{1}{3} c_1 N \big)
= 2 \, e^{- c_2 N}, 
\end{align*}
for appropriate absolute positive constants $c_1, c_2$.
\end{proof}

Below we continue to denote by $c_1,c_2$ the constants from Lemma
\ref{lemma:concentration-l2}, and let $c_3, c_4, \dots$ denote other
absolute positive constants.

\subsection{Proof of Lemma \ref{lemma:principal}}
Let the numbers $q > 2$ and $\eps > 0$, and the real trigonometric
polynomial $u$ (not identically zero) be given. Let $N = N(\eps)$ be
a sufficiently large integer, which will be chosen later. Denote by
$\varphi = \varphi_{q, \gamma}$ the trigonometric polynomial from
Lemma \ref{lemma:auxiliary}. Also let $w = w_{N,u}$ be a real
trigonometric polynomial, satisfying \eqref{eq:w-property} and
\eqref{eq:w-large-l2-norm}, and which has the following additional
property,
\begin{equation}
\label{eq:signed-w-phi}
\text{for every $t \in \T$ either $w(t) u(t) > 0$ or 
otherwise $|w(t)| < c_1 / 2$\,.}
\end{equation}
Such $w$ can be easily found by taking an approximation of the function $\sign(u)$.

Given $0 < \delta < 1$ we use Lemma \ref{lemma:kahane-lemma} to find a measure
$\rho$, supported by the interval $I_0 = \big( \tfrac{1}{4}, \tfrac{1}{3} \big)$,
satisfying \eqref{eq:kahane-constraints} and such that
\begin{equation}
\label{eq:rho-estimate-fixed}
\int |d\rho| < \delta^{-c_3}, \quad \text{where $c_3 := c(I_0)$.}
\end{equation}
Define
\[
\lambda(t) = \int \lambda_s(t) \, d\rho(s).
\]
One can expand the product \eqref{eq:riesz-prod} using the Fourier
representation of the trigonometric polynomial $\varphi$, and this
yields the expression
\[
\lambda(t) = 1 + \sum_{\k} \Big\{ \int s^{l(\k)} d\rho(s) \Big\}
\Big\{ \prod_{k_j \neq 0} \ft{\varphi}(k_j) \Big\}
\; w(t)^{l(\k)} \; e^{i(k_1 \nu + k_2 \nu^2 + \cdots + k_N \nu^N) t},
\]
where the sum goes through all non-zero vectors
\[
\k = (k_1, \dots, k_N) \in \Z^N, \quad |k_j| \leq \deg \varphi,
\]
and $l(\k) > 0$ denotes the number of non-zero coordinates of $\k$.
Note that each polynomial $w(t)^{l(\k)}$ has degree $\leq N \deg w$.
So the condition \eqref{eq:nu-properties} ensures that the summands
in the above sum have disjoint spectra. Taking advantage of the fact
that $\|w(t)^{l(\k)}\|_{A_q} \leq 1$ (which follows from
\eqref{eq:w-property}) we deduce that
\[
\|1 - \lambda\|^q_{A_q}
< \delta^q \sum_{\k} \prod_{k_j \neq 0} |\ft{\varphi}(k_j)|^q
< \delta^q \, (1 + \|\varphi\|_{A_q}^q)^N
< \delta^q \, \exp( N \|\varphi\|_{A_q}^q).
\]
Using \eqref{eq:auxiliary} this implies
\begin{equation}
\label{eq:lambda-close-1}
\|1 - \lambda\|_{A_q} < \delta \, \exp\big( \tfrac1{q} \gamma^q N \big).
\end{equation}

Now consider the trigonometric polynomial $X$ defined in
\eqref{eq:def-x}. Set
\[
E := \{t \in \T : X(t) \geq c_1\}
\quad \text{and} \quad
h := \lambda \cdot \1_{E} \, ,
\]
then
\[
\|\lambda - h\|_{A_q} \leq \|\lambda - h\|_{L^2(\T)}
= \|\lambda\|_{L^2(\T \setminus E)}
\leq \int \|\lambda_s\|_{L^2(\T \setminus E)} \; |d\rho(s)|.
\]
Using Lemma \ref{lemma:concentration-l2} and \eqref{eq:rho-estimate-fixed}
this implies
\begin{equation}
\label{eq:lambda-close-h}
\|\lambda - h\|_{A_q} \leq \sqrt{2} \, e^{-\frac1{2} c_2 N} \, \delta^{-c_3}.
\end{equation}

Let $c_4 > 0$ be an absolute constant so small such that, setting
$\delta := e^{- c_4 N}$, the right hand side of \eqref{eq:lambda-close-h}
will tend to zero as $N \to \infty$. Next, let the number $\gamma > 0$
be an absolute constant, so small such that also the right hand side of
\eqref{eq:lambda-close-1} will tend to zero as $N \to \infty$. Now we fix
$N = N(\eps)$ so large, such that the right hand sides of both
\eqref{eq:lambda-close-1} and \eqref{eq:lambda-close-h} will be smaller
than $\eps / 2$. Having fixed $N$, the functions $w$, $\lambda$, $X$
and $h$ are also fixed, and it follows that
\[
\|1 - h\|_{A_q} \leq \|1 - \lambda\|_{A_q} + \|\lambda - h\|_{A_q} < \eps.
\]

Finally we will define the compact $K$, the function $f$ and the
trigonometric polynomial $P$ with the properties \ref{cond:principal-1}
and \ref{cond:principal-2}. Let $\chi$ be a non-negative, infinitely smooth
function, with integral $= 1$. Set $f := h \ast \chi$, then $\|1 - f\|_{A_q}
< \eps$. By choosing $\chi$ supported on a sufficiently small neighborhood of
zero, we may assume that $f$ is supported by a compact $K$ (a finite union
of segments) such that $X(t) > c_1 / 2$ on $K$. Thus \ref{cond:principal-1}
is satisfied. Now we set
\[
P(t) := (2/c_1) \cdot \frac{1}{N} \sum_{j=1}^{N} \varphi(\nu^j t),
\]
and check that \ref{cond:principal-2} is satisfied.
First, due to \eqref{eq:w-property} we have
\[
|P(t)| \geq P(t) \, w(t) = (2/c_1) \, X(t) > 1, \quad t \in K.
\]
Secondly, since $\|\varphi\|_\infty \leq 1$, for every $t \in K$
we have $|w(t)| \geq X(t) > c_1/2$, and \eqref{eq:signed-w-phi}
implies that $w(t) u(t) > 0$. Hence $P(t) u(t) > 0$ on $K$.
Lastly, 
\[
\|P\|_A \leq (2/c_1) \, \|\varphi\|_A = C(q),
\]
and our main lemma is proved. \qed


\section{Helson sets and distributions}
\label{section:helson}

\subsection{}
\label{subsection:helson-intro}
Recall the main two properties of Piatetski-Shapiro's compact $K$ :

\begin{enumerate-math}
\renewcommand\theenumi{(\Roman{enumi})}
\item
\label{cond:ps-1}
$K$ supports a non-zero distribution $S$ with $\ft{S}(n) \to 0$
as $|n| \to \infty$.
\item
\label{cond:ps-2}
For every non-zero measure $\mu$ supported by $K$,
\[
\limsup_{|n| \to \infty}  |\ft{\mu}(n)| > 0.
\]
\end{enumerate-math}

In a way, the existence of such a compact reveals a ``compromise''
between certain ``thickness'' and ``thinness'' conditions of a set
(understood not in a metrical but rather an arithmetical sense).
We will see that this compromise can be achieved under stronger
conditions, in both directions. 

\begin{definition*}[see for example \cite{kahane:absolument}, Chapter IV]
A compact set $K$ is called a \emph{Helson set} if it satisfies any
one of the following equivalent conditions:
\begin{enumerate-math}
\item
Every continuous function on $K$ admits extension to a function in $A(\T)$.
\item
There is $\delta_1(K) > 0$ such that, for every measure $\mu$
supported by $K$,
\begin{equation}
\label{eq:helson-condition-sup}
\sup_{n \in \Z} |\ft{\mu}(n)| \geq \delta_1(K) \int |d\mu|.
\end{equation}
\item
There is $\delta_2(K) > 0$ such that, for every measure $\mu$
supported by $K$,
\begin{equation}
\label{eq:helson-condition-limsup}
\limsup_{|n| \to \infty} |\ft{\mu}(n)| \geq \delta_2(K) \int |d\mu|.
\end{equation}
\end{enumerate-math}
\end{definition*}

K\"{o}rner \cite{korner} and Kaufman \cite{kaufman} generalized
Piatetski-Shapiro's result by constructing Helson sets with the
property \ref{cond:ps-1} above (that is, Helson M-sets).

We will prove the following stronger theorem:

\begin{theorem}
\label{thm:main-helson}
For any $q > 2$ there is a Helson set $K$ on the circle $\T$, which
supports a non-zero distribution $S$ such that $\ft{S} \in \ell^q$.
\end{theorem}

Clearly this also implies Theorem \ref{thm:piatetski-lq}.

\subsection{}
For the proof of Theorem \ref{thm:main-helson} we need the following

\begin{lemma}
\label{lemma:helson-sufficient-condition}
Let $K$ be a totally disconnected compact set on $\T$. Suppose
that there is a constant $C > 0$ such that the following is true:
given any real-valued function $h \in C(\T)$ with no zeros in $K$,
one can find a real trigonometric polynomial $P(t)$ such that
\begin{equation}
\label{eq:helson-sufficient-condition}
\inf_{t \in K} |P(t)| > 1, \quad
\text{$P(t) h(t) > 0$ on $K$, \quad $\|P\|_A \leq C$.}
\end{equation}
Then $K$ is a Helson set.
\end{lemma}

\begin{proof}
It would be enough to show that there is $\delta_1(K) > 0$,
such that \eqref{eq:helson-condition-sup} is satisfied by every
measure $\mu$ supported by $K$. In fact, it is enough to prove
\eqref{eq:helson-condition-sup} only for real, signed measures $\mu$,
as one can check easily by decomposing a complex measure into its real
and imaginary parts.

Let therefore $\mu$ be a real, signed measure supported by $K$, and
suppose that $\int |d\mu| = 1$. Since $K$ is totally disconnected,
given $\eps > 0$ there is a real-valued function $h \in C(\T)$
such that $h(t) = \pm 1$ on $K$, and $\int h \, d\mu > 1 - \eps$.
Let $P(t)$ be a real trigonometric polynomial satisfying
\eqref{eq:helson-sufficient-condition}. Then
\[
\int_K P \, d\mu = \int_K P \, h \, |d\mu| - 
\int_K P \, h \, (|d\mu| - h \, d\mu)
> 1 - C \eps.
\]
On the other hand,
\[
\int_K P \, d\mu = \int_{\T} P \, d\mu = \sum_{n \in \Z}
\ft{P}(-n) \, \ft{\mu}(n) \leq C \, \sup_{n \in \Z} |\ft{\mu}(n)|.
\]
Since $\eps$ was arbitrary, this shows that \eqref{eq:helson-condition-sup}
is true with $\delta_1(K) = C^{-1}$.
\end{proof}

\begin{remark*}
One can show that the condition in Lemma
\ref{lemma:helson-sufficient-condition} is also necessary for Helson
sets. For comparison, we mention another necessary and sufficient
condition in a similar spirit: a compact $K$ is a Helson set
if and only if it is totally disconnected, and every $\{0,1\}$-valued
continuous function on $K$ admits an extension to $\T$
with bounded $A(\T)$ norm (see \cite{kahane:absolument}, p. 52).
\end{remark*}

\subsection{Proof of Theorem \ref{thm:main-helson}}
Fix $q > 2$. Choose a sequence $u_j$ of real, non-zero trigonometric
polynomials, which is dense in the metric space of real-valued continuous
function on $\T$. For a sequence $\eps_j$ use Lemma \ref{lemma:principal}
with $\eps = \eps_j$ and $u = u_j$ to choose $K_j$, $f_j$ and $P_j$.
We choose the $\eps_j$ by induction, such that
\[
\eps_1 < 2^{-2} \qquad \textrm{and} \qquad
\|f_1 \cdot f_2 \cdots f_j\|_{A} \; \eps_{j+1} <
2^{-2-j} \quad (j=1,2,\dots).
\]
This condition allows to define a distribution $S \in A_q(\T)$ by
the infinite product $\prod_{j=1}^{\infty} f_j$. Indeed, the partial
products $S_j = f_1 \cdot f_2 \cdots f_j$ satisfy
\[
\|S_{j+1} - S_j\|_{A_q} = \|f_1 \cdots f_j \cdot (f_{j+1} - 1)\|_{A_q} \leq
\|f_1 \cdots f_j\|_{A} \; \eps_{j+1} < 2^{-2-j},
\]
hence the $S_j$ converge in $A_q(\T)$ to a limit $S$. Observe that
$S$ is non-zero, since
\[
\|S - 1\|_{A_q} \leq \sum_{j=0}^{\infty} \|S_{j+1} - S_j\|_{A_q}
< \sum_{j=0}^{\infty} 2^{-2-j} < 1,
\]
and that $S$ is supported by the compact $K := \bigcap_{j=1}^{\infty} K_j$.

On the other hand, we will show that $K$ is a Helson set. It is enough
to check that $K$ satisfies the conditions of Lemma
\ref{lemma:helson-sufficient-condition}. Indeed, for each $j$ we have
\[
\inf_{t \in K} |P_j(t)| > 1, \quad
\text{$P_j(t) u_j(t) > 0$ on $K$,} \quad \|P_j\|_A \leq C(q).
\]
In particular, none of the $u_j$ has a zero in $K$. Since they are dense
in the metric space of real-valued continuous function on $\T$, it follows
that $K$ is totally disconnected. Let now $h \in C(\T)$ be a real-valued
function, with no zeros in $K$. Choose $j$ such that $u_j(t) h(t) > 0$
on $K$, then \eqref{eq:helson-sufficient-condition} is satisfied with
$P = P_j$ and $C = C(q)$. It therefore follows from Lemma
\ref{lemma:helson-sufficient-condition} that $K$ is a Helson set. \qed


\section{Helson sets and cyclic vectors}
\label{section:proof-main}

\subsection{}
The role of Helson sets in our problem is clarified by the following

\begin{lemma}
\label{lemma:helson-cyclic}
Let $K$ be a Helson set on $\T$. Then there is a function
$g \in A(\T)$, vanishing on $K$, which is a cyclic vector
in $A_p(\T)$ for every $p > 1$.
\end{lemma}

For the proof of Lemma \ref{lemma:helson-cyclic} we will need the
following property of Helson sets. Denote by $C(K)$ the space of
continuous functions on $K$ with the norm
\[
\|h\|_{C(K)} = \sup_{t \in K} |h(t)|.
\]
Recall that one of the equivalent definitions of a Helson set is that
every element of $C(K)$ admits an extension to a function in $A(\T)$.
The next lemma shows that one can actually find such extensions with
arbitrarily small $A_p$ norms.

\begin{lemma}
\label{lemma:helson-extension}
Let $K$ be a Helson set, and suppose that $\eps > 0$, $p > 1$ and
$h \in C(K)$ are given. Then one can find $f \in A(\T)$ such that
\[
f|_K = h, \quad \|f\|_A \leq (1/\delta) \, \|h\|_{C(K)}, 
\quad \|f\|_{A_p} < \eps,
\]
where $\delta = \delta_2(K) > 0$ is the constant from
\eqref{eq:helson-condition-limsup}.
\end{lemma}

\begin{proof}
Fix $p > 1$ and $\eps > 0$. Introduce a Banach space $B = B_{p, \eps}$
of functions $f$ on the circle $\T$ such that
\[
\|f\|_B := \|f\|_A + (1/\eps) \, \|f\|_{A_p} < \infty.
\]
In other words, the space $B$ coincides with the space $A(\T)$ but is
equipped with a different (equivalent) norm. Let also $T: B \to C(K)$ be
the restriction operator $f \mapsto f|_K$, and denote by $T^*$ its dual operator.

Given a measure $\mu$ supported by $K$, by \eqref{eq:helson-condition-limsup}
we have
\[
L(\mu) := \limsup_{|n| \to \infty} |\ft{\mu}(n)| \geq \delta \int |d\mu|.
\]
Take a sequence of integers $n_j$, $|n_1| < |n_2| < \cdots$, and real numbers
$\theta_j$ such that
\[
\lim_{j \to \infty} \ft{\mu}(n_j) \, e^{-i \theta_j} = L(\mu),
\]
and define
\[
f_N(t) = \frac1{N} \sum_{j=1}^{N} e^{-i (n_j t + \theta_j)}.
\]
Then $\|f_N\|_B = 1 + (1 / \eps) \, N^{(1/p) - 1}$, and
\[
\dotprod{f_N}{T^* \mu} = \dotprod{T f_N}{\mu} = \int_{K} f_N(t) \, d\mu(t)
= \frac1{N} \sum_{j=1}^{N} \ft{\mu}(n_j) \, e^{-i \theta_j}.
\]
It follows that
\[
\|T^* \mu\|_{B^*} \geq 
\lim_{N \to \infty} \frac{|\dotprod{f_N}{T^* \mu}|}{\|f_N\|_B}
= L(\mu) \geq \delta \int |d\mu|,
\]
for every measure $\mu$ supported by $K$. 

By a classical theorem of Banach (see \cite{kahane-salem}, p. 141)
this implies that for every $h \in C(K)$, the equation $Tf = h$ admits
a solution $f \in B$ such that $\|f\|_B \leq (1 / \delta) \, \|h\|_{C(K)}$.
This proves the lemma.
\end{proof}

\subsection{}
Using Lemma \ref{lemma:helson-extension} we can prove Lemma
\ref{lemma:helson-cyclic} above.

\begin{proof}[Proof of Lemma \ref{lemma:helson-cyclic}]
It will be convenient to use Baire categories in the proof. Let $I(K)$
denote the set of functions $g \in A(\T)$ which vanish on $K$. This is
a complete metric space, with the metric inherited from $A(\T)$. We will
prove that the set of functions $g \in I(K)$ which are cyclic in $A_p(\T)$
for every $p > 1$, is a countable intersection of open, dense sets in
the space $I(K)$. By Baire's theorem, this set is therefore non-empty
(and in fact is dense in the space).

For $\eps > 0$ and $p > 1$, denote by $G(\eps,p)$ the set of $g \in I(K)$
for which there exists a trigonometric polynomial $P$ such that $\|1 - P
\cdot g\|_{A_p} < \eps$. Choose a sequence $\eps_n \to 0$ and a sequence
$p_n \to 1$ $(n \to \infty)$, and consider the intersection
\[
\bigcap_{n=1}^{\infty} G(\eps_n, p_n).
\]
According to condition \ref{remark:polynom-cyclic} from Section
\ref{subsection:background}, a function $g \in I(K)$ belongs to this
intersection if and only if it is cyclic in $A_p(\T)$ for every $p > 1$.
So to conclude the proof it remains to show that each $G(\eps,p)$ is an open,
dense set in $I(K)$.

Let $g_0 \in G(\eps, p)$ be given. Then $\|1 - P \cdot g_0\|_{A_p} < \eps$ for
some trigonometric polynomial $P$. Given $\eta > 0$, suppose that $g \in I(K)$
and $\|g - g_0\|_A < \eta$. Then
\[
\|1 - P \cdot g\|_{A_p} \leq \|1 - P \cdot g_0\|_{A_p} + \eta \, \|P\|_{A_p}.
\]
It $\eta$ is chosen sufficiently small then the right hand side is smaller than
$\eps$. Hence $G(\eps, p)$ contains the open ball $B(g_0, \eta)$ of radius
$\eta$ centered at $g_0$, and this shows that $G(\eps, p)$ is open.

Finally we show that $G(\eps,p)$ is dense. Let a ball $B(g_0, \eta)$ in
$I(K)$ be given. Choose a trigonometric polynomial $h$, not identically zero,
such that
\[
\|h - g_0\|_A < \frac{\delta}{1 + \delta} \cdot \eta \, ,
\]
where $\delta = \delta_2(K) > 0$ is the constant from \eqref{eq:helson-condition-limsup}.
In particular this implies that
\[
\sup_{t \in K} |h(t)| < \frac{\delta}{1 + \delta} \cdot \eta \, .
\]
Since $h$ is non-zero, it has finitely many zeros, so by conditions
\ref{remark:polynom-cyclic} and \ref{remark:finite-cyclic} from Section
\ref{subsection:background} there is a trigonometric polynomial $P$ such
that $\|1 - P \cdot h\|_{A_p} < \eps/2$.
Now use Lemma \ref{lemma:helson-extension} to find $f \in A(\T)$ such that
\[
f|_K = h|_K, \quad \|f\|_A < \eta / (1 + \delta), \quad
\|f\|_{A_p} < \frac{\eps}{2 \, \|P\|_A},
\]
and set $g := h - f$. Then clearly $g \in I(K)$. Moreover
\[
\|g - g_0\|_A \leq \|h - g_0\|_A + \|f\|_A < \eta,
\]
that is, $g \in B(g_0, \eta)$. Also,
\[
\|1 - P \cdot g\|_{A_p} \leq \|1 - P \cdot h\|_{A_p} + \|P\|_A \|f\|_{A_p}
< \eps,
\]
and therefore $g \in G(\eps, p)$. This shows that $G(\eps,p)$ is dense.
\end{proof}

\subsection{}
Our main result now follows:

\begin{proof}[Proof of Theorem \ref{thm:main-cyclic}]
By Theorem \ref{thm:main-helson} there is a Helson set $K$ satisfying
the condition \ref{cond:atag} in Theorem \ref{thm:piatetski-lq}.
This condition is equivalent to condition \ref{cond:a} in Theorem
\ref{thm:main-cyclic} (see Section \ref{subsection:background},
\ref{remark:noncyclic-distribution} and \ref{remark:smooth-distribution-noncyclic}).
On the other hand, Lemma \ref{lemma:helson-cyclic} implies that $K$
satisfies also condition \ref{cond:b}. So Theorem \ref{thm:main-cyclic}
is proved.
\end{proof}

\begin{proof}[Proof of Corollary \ref{cor:impossible-integers}]
Let $K$ be the compact set of Theorem \ref{thm:main-cyclic}. By the
property \ref{cond:b} there is $g \in A(\T)$ vanishing on $K$, which
is cyclic in $A_p(\T)$. Choose a smooth (say, twice continuously
differentiable) function $f$ on $\T$, such that $\zf = \zg$.
In particular, $f$ vanishes on $K$. Since the Fourier coefficients of
$f$ decrease sufficiently fast, the property \ref{cond:a} implies that
$f$ is a non-cyclic vector in $A_p(\T)$. Thus our corollary is proved.
\end{proof}

\begin{remark*}
One can see from the proof above that if $p$ remains bounded away
from $2$, then the two vectors in Corollary \ref{cor:impossible-integers}
may be chosen independently of $p$.
\end{remark*}


\section{Non-periodic version}
\label{section:non-periodic}

Here we extend the results to $L^p(\R)$ spaces, $1 < p < 2$.
This can be deduced easily from the previous results, so we may be brief.
In particular we skip the formulation of the corresponding version of Theorem
\ref{thm:main-cyclic}, and restrict ourselves to

\begin{corollary}
\label{cor:impossible-reals}
Given any $1 < p < 2$ one can find two functions in $L^1(\R) \cap C_0(\R)$,
such that one is cyclic in $L^p(\R)$ and the other is not, but their
Fourier transforms have the same (compact) set of zeros.
\end{corollary}

Here $C_0(\R)$ is the space of continuous functions on $\R$ vanishing
at infinity.

It will be convenient to denote by $\hat{\R}$ another copy of the real
line. We consider distributions on the Schwartz space $S(\hat{\R})$.
We denote by $A_p(\hat{\R})$, $1 \leq p < \infty$, the space of Fourier
transforms of functions in $L^p(\R)$, with the corresponding norm.
In particular, for $p = 1$ this is the Wiener algebra $A(\hat{\R})$
of functions with an absolutely convergent Fourier integral.

Recall that, by definition, a function $F(x) \in L^p(\R)$ is a cyclic
vector if the translates $\{F(x - y)\}$, $y \in \R$, span the whole space.
Equivalently, $F$ is cyclic if the set $\{f(t) \, \phi(t)\}$, where
$f = \ft{F}$ and $\phi$ runs over $S(\hat{\R})$, is dense in $A_p(\hat{\R})$.

\begin{proof}[Proof of Corollary \ref{cor:impossible-reals}]
Fix $1<p<2$, and take the compact $K$ of Theo\-rem~\ref{thm:main-cyclic}.
By the property \ref{cond:b} there is $h \in A(\T)$, vanishing on $K$,
which is cyclic in $A_p(\T)$ (with respect to multiplication by
trigonometric polynomials).

We may assume that $h(t)$ is positive at some point $t$, and by rotation,
that $h(\pi) > 0$. It follows that there is an interval
\[
I := (-\pi + \delta, \; \, \pi - \delta), \quad \delta > 0,
\]
such that $K \subset \{t : \Re h(t) \leq 0\} \subset I$. Choose a
function $\chi \in S(\hat{\R})$, $0 \leq \chi \leq 1$, compactly
supported by $(-\pi, \pi)$, and such that $\chi(t) = 1$ on $I$. Define
\[
g(t) := \chi(t) \, h(t) + (1 - \chi(t)) \, e^{-t^2}, \quad t \in \hat{\R}.
\]
It is easy to see that:
\begin{enumerate-math}
\item
The zero set $\zg$ is compact, $K \subset \zg \subset I$.
\item
$g = \ft{G}$ for some $G \in L^1(\R) \cap C_0(\R)$,
which implies $g \in A_p(\hat{\R})$.
\end{enumerate-math}

\subsubsection*{Claim 1}
The set $\{g(t) \, \phi(t)\}$,
$\phi \in S(\hat{\R})$, is dense in $A_p(\hat{\R})$.

If not then, by duality, there is a (non-zero) distribution
$S \in A_q(\hat{\R})$, $q = p/(p-1)$, such that 
$\dotprod{S}{\; g \cdot \phi} = 0$ for every $\phi \in S(\hat{\R})$.
It follows that
\begin{equation}
\label{eq:support-interval}
\supp(S) \subset \zg \subset I, \quad |I| < 2 \pi.
\end{equation}
We have $g(t) = h(t)$ on $I$, since $\chi(t) = 1$ on $I$. Hence
\begin{equation}
\label{eq:annihilate}
\dotprod{S}{\; h \cdot \phi} = 0
\quad \text{for every $\phi \in S(\hat{\R})$.}
\end{equation}
The condition \eqref{eq:support-interval} allows us to regard $S$ also
as a distribution on $\T$. It is well known that under this condition
the following equivalence holds:
\[
S \in A_q(\T) \; \iff \; S \in A_q(\hat{\R}).
\]
But $h$ is cyclic in $A_p(\T)$, so \eqref{eq:annihilate} implies
that $S = 0$, which proves the claim.

Now take an arbitrary function $f \in S(\hat{\R})$ with $\zf = \zg$.

\subsubsection*{Claim 2}
The set $\{f(t) \, \phi(t)\}$,
$\phi \in S(\hat{\R})$, is not dense in $A_p(\hat{\R})$.

Indeed, the property \ref{cond:a} from Theorem \ref{thm:main-cyclic}
implies that $K$ supports a (non-zero) distribution $S \in A_q(\T)$.
As above we can regard it as a distribution on $\hat{\R}$, belonging
to $A_q(\hat{\R})$. But $f$ is a \emph{smooth} function in $A_p(\hat{\R})$,
and $f|_K=0$, hence $\dotprod{S}{\; f \cdot \phi} = 0$ for every
$\phi \in S(\hat{\R})$.

This means that the inverse Fourier transform of $f$ is a function
$F \in L^1(\R) \cap C_0(\R)$, which is non-cyclic in $L^p(\R)$.
Our corollary is thus proved.
\end{proof}


\section{Remarks}
\label{section:remarks}

\subsection{}
Theorem \ref{thm:main-cyclic} may be put into the context of the
theory of translation-invariant subspaces. A linear subspace
$M \subset L^p(G)$ is called translation-invariant if whenever $f$
belongs to $M$, then so do all of the translates of $f$. Observe that
$f \in L^p(G)$ is a cyclic vector if and only if it does not belong
to any proper closed translation-invariant subspace of $L^p(G)$.

It is well known that any closed translation-invariant subspace in 
$\ell^2(\Z)$ can be uniquely recovered from the set of the common zeros
of the Fourier transforms of its elements.

This is not the case in $\ell^1(\Z)$. Malliavin's ``non-synthesis'' example
\cite{malliavin} means that different closed translation-invariant subspaces
in this space may have the same set of common zeros. More precisely,
for a compact set $K \subset \T$ consider the invariant subspaces
\[
I(K) = \{ \c \in \ell^1(\Z) \, : \, \text{$\ft{\c}$ vanishes on $K$}\},
\]
\[
J(K) = \{ \c \in \ell^1(\Z) \, : \, \text{$\ft{\c}$ vanishes on some open
set containing $K$}\}.
\]
Malliavin proved that there is a compact $K$ such that the closure
of $J(K)$ is strictly smaller than $I(K)$.

Kahane \cite{kahane:synthese} (see also \cite{kahane-salem}, p. 121)
showed that such a result still holds if one takes the closures of $J(K)$
and $I(K)$ in $\ell^p(\Z)$, $1 < p < 2$.

Theorem \ref{thm:main-cyclic} reveals a sharper phenomenon in these
spaces, which is not possible in $\ell^1(\Z)$. Namely, there is a compact
$K$ such that the $\ell^p$-closures satisfy
\[
\clos J(K) \subsetneqq \clos I(K) = \ell^p(\Z).
\]

\subsection{}
Strictly speaking, it is not necessary to require that $\c \in \ell^1$
in order to have the zero set $\zftc$ well-defined. The continuity of
$\ft{\c}$ is sufficient for that, as appeared in the weaker version of
Theorem \ref{thm:main-cyclic} proved in \cite{lev-olevskii:generators}.
However, the advantage of the present version seems to be substantial,
since very little is known on the relation between cyclicity in
$\ell^p$ $(1<p<2)$ and the zero set unless $\c \in \ell^1$.
In particular we do not know the answer to the following question:
let $f \in C(\T) \cap A_p(\T)$ have no zeros, does this imply that
$\c = \ft{f}$ is a cyclic vector?

\subsection{}
Let $K$ be a Helson set. Define its `Helson constant' as the maximal
possible $\delta_1(K)$ in \eqref{eq:helson-condition-sup}. K\"{o}rner's
\cite{korner} and Kaufman's \cite{kaufman} constructions (see Section
\ref{subsection:helson-intro} above) give a Helson constant $1$. What can
be said about the Helson constant of $K$ in Theorem \ref{thm:main-helson}\,?
Specifically, must it tend to zero when $q \to 2$\,?

\subsection{}
Perhaps the most interesting problem left open is:
could one characterize in reasonable terms the cyclic vectors $\c$
in $\ell^p$, $1<p<2$, under the standard assumption $\c \in \ell^1$
with no extra restrictions?


\end{document}